\documentclass[10pt]{amsart}%
\usepackage{amsfonts}
\usepackage{amsmath}
\usepackage{amssymb}
\usepackage{graphicx}%
\providecommand{\U}[1]{\protect\rule{.1in}{.1in}}
\newtheorem{theorem}{Theorem}
\theoremstyle{plain}

\newtheorem{corollary}{Corollary}

\newtheorem{definition}{Definition}

\newtheorem{lemma}{Lemma}

\newtheorem{proposition}{Proposition}

\numberwithin{equation}{section}
\begin{document}
\title[Absolutely $\gamma$-summing multilinear operators]{Absolutely $\gamma$-summing multilinear operators}
\author[D.M. Serrano-Rodr\'{\i}guez]{Diana Marcela Serrano-Rodr\'{\i}guez\textsuperscript{*}}
\address{Departamento de Matem\'{a}tica, UFPB, Jo\~{a}o Pessoa, PB, Brazil}
\email{dmserrano0@gmail.com}
\thanks{\textsuperscript{*}Supported by Capes.}

\begin{abstract}
In this paper we introduce an abstract approach to the notion of absolutely
summing multilinear operators. We show that several previous results on
different contexts (absolutely summability, almost summability, Cohen
summability) are particular cases of our general results.

\end{abstract}
\maketitle

\section{Introduction and background}

A linear operator $u:E\rightarrow F$ is absolutely summing if ${\textstyle\sum
}u(x_{j})$ is absolutely convergent whenever ${\textstyle\sum}x_{j}$ is
unconditionally convergent. Thanks to Grothendieck we know that every
continuous linear operator from $\ell_{1}$ to any Hilbert space is absolutely
summing. This result is one of the several consequences of the famous
Grothendieck's inequality, called \textquotedblleft the fundamental theorem of
the metric theory of tensor products\textquotedblright.

The more general notion of absolutely $(p;q)$-summing operators was introduced
in the 1960's by B. Mitiagin and A. Pe\l czy\'{n}ski \cite{MiP} and A. Pietsch
\cite{stu}. In \cite{MSRam} an abstract approach to absolutely summing
operators, for very general sequence spaces was introduced and explored. This
class was called the class of absolutely $\lambda$-summing operators. The
search of abstract environments where a more general theory holds has also
been investigated in different papers, also for non-multilinear operators (we
mention \cite{MST, adv, lon} and the references therein).

In this paper we introduce a similar procedure for the multilinear setting. We
show that various well-known multilinear results (and also some other new
results) are particular cases of our approach.

Henceforth $\mathbb{N}$ represents the set of all positive integers and $E$,
$E_{1},\ldots,E_{n}$, $F$, $G$, $G_{1},\ldots,G_{n}$, and $H$ will stand for
Banach spaces over $\mathbb{K}=\mathbb{R}$ or $\mathbb{C}$. The topological
dual of $E$ is represented by $E^{\ast}$ and $B_{E^{\ast}}$ denotes its closed
unit ball.

\section{Absolutely summing multilinear operators}

The multilinear approach to absolutely summing operators was initiated by
Pietsch and followed by several authors (see \cite{BBJ, matos, pgtese} and the
references therein). The following concept was introduced by M.C. Matos
(\cite{matos}):

\begin{definition}
[Absolutely summing multilinear operators.]\label{mas}If $\frac{1}{p}\leq
\frac{1}{q_{1}}+...+\frac{1}{q_{m}}$ a multilinear operator $T\in
\mathcal{L}\left(  E_{1},...,E_{m};F\right)  $ is absolutely $\left(
p;q_{1},...,q_{m}\right)  $-summing at the point $a=\left(  a_{1}%
,...,a_{m}\right)  \in E_{1}\times...\times E_{m}$ when%
\[
\left(  T\left(  a_{1}+x_{j}^{1},...,a_{m}+x_{j}^{m}\right)  -T\left(
a_{1},...,a_{m}\right)  \right)  _{j=1}^{\infty}\in\ell_{p}\left(  F\right)
\]
for all $\left(  x_{j}^{k}\right)  _{j=1}^{\infty}\in\ell_{q_{k}}^{w}\left(
E_{k}\right)  .$This class is denoted by $\mathcal{L}_{as,\left(
p;q_{1},...,q_{m}\right)  }^{a}.$ When $a$ is the origin we just write
$\mathcal{L}_{as,\left(  p;q_{1},...,q_{m}\right)  }$, and if $T$ is
absolutely $\left(  p;q_{1},...,q_{m}\right)  $-summing everywhere we write
$T\in\mathcal{L}_{as,\left(  p;q_{1},...,q_{m}\right)  }^{ev}.$\bigskip
\end{definition}

The following characterization of $\mathcal{L}_{as,\left(  p;q_{1}%
,...,q_{m}\right)  }$ is folklore, see \cite[\textit{Theorem }$1.2$]{botcot}.

\begin{proposition}
\label{p9}$T\in\mathcal{L}_{as,\left(  p;q_{1},...,q_{m}\right)  }\left(
E_{1},...,E_{m};F\right)  $ if only if, there exists a constant $C>0$ such
that%
\begin{equation}
\left\Vert \left(  T\left(  x_{j}^{\left(  1\right)  },...,x_{j}^{\left(
m\right)  }\right)  \right)  _{j=1}^{\infty}\right\Vert _{p}\leq C\prod
_{s=1}^{m}\left\Vert \left(  x_{j}^{\left(  s\right)  }\right)  _{j=1}%
^{\infty}\right\Vert _{w,q_{s}} \label{3b}%
\end{equation}
for all $\left(  x_{j}^{\left(  s\right)  }\right)  _{j=1}^{\infty}\in
\ell_{q_{s}}^{w},$ $s=1,...,m.$ In addition, the smallest of the constants $C$
satisfying (\ref{3b}), denoted by $\left\Vert \cdot\right\Vert _{as}$, defines
a norm on $\mathcal{L}_{as,\left(  p;q_{1},...,q_{m}\right)  }\left(
E_{1},...,E_{m};F\right)  $.
\end{proposition}

For everywhere absolutely summing operators we also have a characterization by
inequalities (\cite[\textit{Theorem} 4.1]{pellscan}):

\begin{theorem}
For $T\in\mathcal{L}\left(  E_{1},...,E_{m};F\right)  $, the following
statements are equivalent$:$

$\left(  i\right)  $ $T\in\mathcal{L}_{as,\left(  p;q_{1},...,q_{m}\right)
}^{ev}\left(  E_{1},...,E_{m};F\right)  ;$

$\left(  ii\right)  ~$There is a constant $C>0$ such that%
\begin{align*}
&  \left\Vert \left(  T\left(  b_{1}+x_{j}^{\left(  1\right)  },...,b_{m}%
+x_{j}^{\left(  m\right)  }\right)  -T\left(  b_{1},...,b_{m}\right)  \right)
_{j=1}^{n}\right\Vert _{p}\\
&  \leq C\left(  \left\Vert b_{1}\right\Vert +\left\Vert \left(
x_{j}^{\left(  1\right)  }\right)  _{j=1}^{n}\right\Vert _{w,q_{1}}\right)
\cdots\left(  \left\Vert b_{m}\right\Vert +\left\Vert \left(  x_{j}^{\left(
m\right)  }\right)  _{j=1}^{n}\right\Vert _{w,q_{m}}\right)  ,
\end{align*}
for all $n.$

$\left(  iii\right)  $ There is a constant $C>0$ such that%
\begin{align}
&  \left\Vert \left(  T\left(  b_{1}+x_{j}^{\left(  1\right)  },...,b_{m}%
+x_{j}^{\left(  m\right)  }\right)  -T\left(  b_{1},...,b_{m}\right)  \right)
_{j=1}^{\infty}\right\Vert _{p}\label{d2}\\
&  \leq C\left(  \left\Vert b_{1}\right\Vert +\left\Vert \left(
x_{j}^{\left(  1\right)  }\right)  _{j=1}^{\infty}\right\Vert _{w,q_{1}%
}\right)  \cdots\left(  \left\Vert b_{m}\right\Vert +\left\Vert \left(
x_{j}^{\left(  m\right)  }\right)  _{j=1}^{\infty}\right\Vert _{w,q_{m}%
}\right)  ,\nonumber
\end{align}
for all $\left(  b_{1},...,b_{m}\right)  \in E_{1}\times\cdots\times E_{m}$
and $\left(  x_{j}^{\left(  r\right)  }\right)  _{j=1}^{\infty}\in\ell_{q_{r}%
}^{w},~r=1,...,m.$
\end{theorem}

Moreover, the smallest $C$ such that (\ref{d2}) is satisfied, denoted by
$\left\Vert \cdot\right\Vert _{ev}$, defines a norm on $\mathcal{L}_{as}%
^{ev}\left(  E_{1},...,E_{m};F\right)  .$ In both cases $\left(
\mathcal{L}_{as,\left(  p;q_{1},...,q_{m}\right)  }\left(  E_{1}%
,...,E_{m};F\right)  ,\left\Vert \cdot\right\Vert _{as}\right)  $ and $\left(
\mathcal{L}_{as,\left(  p;q_{1},...,q_{m}\right)  }^{ev}\left(  E_{1}%
,...,E_{m};F\right)  ,\left\Vert \cdot\right\Vert _{ev}\right)  $ are a Banach spaces.

\bigskip

In this paper we consider a quite general version of the notion of everywhere
absolutely summing operators. We work with quite arbitrary sequence spaces
instead of $\ell_{q_{r}}^{w}\left(  E_{r}\right)  $ and $\ell_{p}(F).$ Our
results encompass several particular approaches found in the literature, as
detailed in Section \ref{appl}.

\section{Absolutely $\gamma_{\left(  s;s_{1},...,s_{m}\right)  }$-summing
multilinear operators}

In the 70's M. Ramanujan introduced an abstract approach to the notion of
absolutely summing operators \cite{MSRam}. In \cite{MSRam} some sequence
spaces, denoted by $\lambda\left(  E\right)  $ and $\lambda\left[  F\right]  $
are considered (the case $\ell_{p}=\lambda$ is a particular case, with
$\lambda\left(  E\right)  =\ell_{p}^{w}\left(  E\right)  $ and $\lambda\left[
F\right]  =\ell_{p}\left(  F\right)  $). A bounded linear operator
$T:E\rightarrow F$ is absolutely $\lambda$-summing if, for each $x=\left(
x_{i}\right)  _{i=1}^{\infty}\in\lambda\left(  E\right)  $, we have $T\left(
x\right)  =\left(  Tx_{i}\right)  \in\lambda\left[  F\right]  $. For our
purposes the following simple definition is sufficient:

\begin{definition}
Let $E$ be a Banach space. A sequence space in $E$ is a vector space
$\gamma\left(  E\right)  \subset E^{\mathbb{N}}$ with a complete norm
$\left\Vert \cdot\right\Vert _{\gamma\left(  E\right)  }$. The following
properties on $\gamma\left(  E\right)  $ are tacitly assumed to hold:

(P1) $\left\Vert \left(  x_{k}\right)  _{k=1}^{\infty}\right\Vert
_{\gamma\left(  E\right)  }=\sup_{n}\left\Vert \left(  x_{k}\right)
_{k=1}^{n}\right\Vert _{\gamma\left(  E\right)  }.$

(P2) $\left\Vert \left(  x_{n}\right)  _{n=1}^{\infty}\right\Vert
_{\gamma\left(  E\right)  }=\left\Vert x_{k}\right\Vert _{E}$, for all
$\left(  x_{n}\right)  _{n=1}^{\infty}=\left(  0,...,0,x_{k},0,...\right)  .$
\end{definition}

The following definition is the natural abstract approach to summability at a
given point.

\begin{definition}
Let $m\in\mathbb{N}$ and $E_{1},...,E_{m}$,$F$ be Banach spaces. An operator
$T\in\mathcal{L}\left(  E_{1},...,E_{m};F\right)  $ is $\gamma_{\left(
s;s_{1},...,s_{m}\right)  }$-summing at $\left(  a_{1},...,a_{m}\right)  \in
E_{1}\times\cdots\times E_{m}$ when%
\[
\left(  T\left(  a_{1}+x_{j}^{\left(  1\right)  },...,a_{m}+x_{j}^{\left(
m\right)  }\right)  -T\left(  a_{1},...,a_{m}\right)  \right)  _{j=1}^{\infty
}\in\gamma_{s}\left(  F\right)
\]
whenever $\left(  x_{j}^{\left(  r\right)  }\right)  _{j=1}^{\infty}\in
\gamma_{s_{r}}\left(  E_{r}\right)  ,$ $r=1,...,m.$
\end{definition}

\bigskip

We denote the space of the $m$-linear operators from $E_{1}\times\cdots\times
E_{m}$ to $F$ which are $\gamma_{\left(  s;s_{1},...,s_{m}\right)  }$-summing
at $\left(  a_{1},...,a_{m}\right)  $ by $\Pi_{\gamma_{\left(  s;s_{1}%
,...,s_{m}\right)  }}^{a}\left(  E_{1},...,E_{m};F\right)  .~$It is plain that
$\Pi_{\gamma_{\left(  s;s_{1},...,s_{m}\right)  }}^{a}\left(  E_{1}%
,...,E_{m};F\right)  $ is a linear subspace of $\mathcal{L}\left(
E_{1},...,E_{m};F\right)  .$ When $a=\left(  0,...,0\right)  $ we write
$\Pi_{\gamma_{\left(  s;s_{1},...,s_{m}\right)  }}\left(  E_{1},...,E_{m}%
;F\right)  $ and when $T$ is $\gamma_{\left(  s;s_{1},...,s_{m}\right)  }%
$-summing at all $\left(  a_{1},...,a_{m}\right)  $ we write $\Pi
_{\gamma_{\left(  s;s_{1},...,s_{m}\right)  }}^{ev}\left(  E_{1}%
,...,E_{m};F\right)  .$ When $s_{1}=\cdots=s_{m}$, we write $\Pi
_{\gamma_{\left(  s;s_{1}\right)  }}\left(  E_{1},...,E_{m};F\right)  $, or
$\Pi_{\gamma_{\left(  s;s_{1}\right)  }}^{a}\left(  E_{1},...,E_{m};F\right)
$, or $\Pi_{\gamma_{\left(  s;s_{1}\right)  }}^{ev}\left(  E_{1}%
,...,E_{m};F\right)  ,$ respectively.\bigskip

The following result is an abstract version of Proposition \ref{p9}:

\begin{proposition}
\label{esse} $T\in\Pi_{\gamma_{\left(  s;s_{1},...,s_{m}\right)  }}\left(
E_{1},...,E_{m};F\right)  $ if and only if there exists a constant $C>0,$ such
that%
\begin{equation}
\left\Vert \left(  T\left(  x_{j}^{\left(  1\right)  },...,x_{j}^{\left(
m\right)  }\right)  \right)  _{j=1}^{\infty}\right\Vert _{\gamma_{s}\left(
F\right)  }\leq C\prod_{r=1}^{m}\left\Vert \left(  x_{j}^{\left(  r\right)
}\right)  _{j=1}^{\infty}\right\Vert _{\gamma_{s_{r}}\left(  E_{r}\right)  }
\label{3}%
\end{equation}
for all $\left(  x_{j}^{\left(  r\right)  }\right)  _{j=1}^{\infty}\in
\gamma_{s_{r}}\left(  E_{r}\right)  ,$ $r=1,...,m.$ Moreover, the smallest $C$
such that (\ref{3}) is satisfied, denoted by $\pi\left(  \cdot\right)  $,
defines a norm on $\Pi_{\gamma_{\left(  s;s_{1},...,s_{m}\right)  }}\left(
E_{1},...,E_{m};F\right)  $.
\end{proposition}

\begin{proof}
Consider $\hat{T}:\gamma_{s_{1}}\left(  E_{1}\right)  \times\cdots\times
\gamma_{s_{m}}\left(  E_{m}\right)  \rightarrow\gamma_{s}\left(  F\right)  $
given by%
\[
\hat{T}\left(  \left(  x_{j}^{\left(  1\right)  }\right)  _{j=1}^{\infty
},...,\left(  x_{j}^{\left(  m\right)  }\right)  _{j=1}^{\infty}\right)
=\left(  T\left(  x_{j}^{\left(  1\right)  },...,x_{j}^{\left(  m\right)
}\right)  \right)  _{j=1}^{\infty}.
\]
Let $\left(  x_{n}\right)  _{n=1}^{\infty}:=\left(  \left(  \left(
x_{n,j}^{\left(  1\right)  }\right)  _{j=1}^{\infty},...,\left(
x_{n,j}^{\left(  m\right)  }\right)  _{j=1}^{\infty}\right)  \right)
_{n=1}^{\infty}$ be a sequence in $\gamma_{s_{1}}\left(  E_{1}\right)
\times\cdots\times\gamma_{s_{m}}\left(  E_{m}\right)  $ converging to
\[
x=\left(  \left(  x_{j}^{\left(  1\right)  }\right)  _{j=1}^{\infty
},...,\left(  x_{j}^{\left(  m\right)  }\right)  _{j=1}^{\infty}\right)
\]
and such that
\[
\lim_{n\rightarrow\infty}\hat{T}(x_{n})=\left(  z_{j}\right)  _{j=1}^{\infty
}\in\gamma_{s}\left(  F\right)  .
\]
Let us prove that $\left(  z_{j}\right)  _{j=1}^{\infty}=\hat{T}\left(
x\right)  .$ Since
\begin{equation}
\lim_{n\rightarrow\infty}\left(  \hat{T}\left(  \left(  x_{n,j}^{\left(
1\right)  }\right)  _{j=1}^{\infty},...,\left(  x_{n,j}^{\left(  m\right)
}\right)  _{j=1}^{\infty}\right)  \right)  =\lim_{n\rightarrow\infty}\left(
T\left(  x_{n,j}^{\left(  1\right)  },...,x_{n,j}^{\left(  m\right)  }\right)
\right)  _{j=1}^{\infty}=\left(  z_{j}\right)  _{j=1}^{\infty},\nonumber
\end{equation}
from \textit{(P1) }and\textit{\ (P2)} we have
\begin{equation}
\lim_{n\rightarrow\infty}T\left(  x_{n,j}^{1},...,x_{n,j}^{m}\right)
=z_{j}\text{, } \label{dd2}%
\end{equation}

and
\[
\lim_{n\rightarrow\infty}x_{n,j}^{\left(  r\right)  }=x_{j}^{\left(  r\right)
}%
\]
for all $j\in\mathbb{N}$, and $r=1,...,m$. Thus, from the continuity of $T$,
it follows that%
\begin{equation}
\lim_{n\rightarrow\infty}T\left(  x_{n,j}^{\left(  1\right)  },...,x_{n,j}%
^{\left(  m\right)  }\right)  =T\left(  x_{j}^{\left(  1\right)  }%
,...,x_{j}^{\left(  m\right)  }\right)  \label{2fevc}%
\end{equation}
for all $j\in\mathbb{N}$. From (\ref{dd2}) and (\ref{2fevc}) we have%
\begin{equation}
T\left(  x_{j}^{\left(  1\right)  },...,x_{j}^{\left(  m\right)  }\right)
=z_{j} \label{2fevf}%
\end{equation}
for all $j\in\mathbb{N}$. Hence%
\[
\lim_{n\rightarrow\infty}\hat{T}\left(  \left(  x_{n,j}^{\left(  1\right)
}\right)  _{j=1}^{\infty},...,\left(  x_{n,j}^{\left(  m\right)  }\right)
_{j=1}^{\infty}\right)  =\hat{T}\left(  \left(  x_{j}^{\left(  1\right)
}\right)  _{j=1}^{\infty},...,\left(  x_{j}^{\left(  m\right)  }\right)
_{j=1}^{\infty}\right)
\]
and $\hat{T}$ is continuous and the rest of the proof follows
straightforwardly. The converse is simple and note that $\pi\left(  T\right)
=\left\Vert \hat{T}\right\Vert .$
\end{proof}

\bigskip

The following lemma can be proved following the lines of
\ \cite[\textit{Lemma} $9.2$]{studia}.

\begin{lemma}
\label{9.2}If $T\in\Pi_{\gamma_{\left(  s;s_{1},...,s_{m}\right)  }}\left(
E_{1},...,E_{m};F\right)  $ and $\left(  a_{1},...,a_{m}\right)  \in
E_{1}\times\cdots\times E_{m}$, then there is a constant $C_{a_{1},...,a_{m}%
}\geq0,$ such that%
\[
\left\Vert \left(  T\left(  a_{1}+x_{j}^{\left(  1\right)  },...,a_{m}%
+x_{j}^{\left(  m\right)  }\right)  -T\left(  a_{1},...,a_{m}\right)  \right)
_{j=1}^{\infty}\right\Vert _{\gamma_{s}\left(  F\right)  }\leq C_{a_{1}%
,...,a_{m}}%
\]
for all $\left(  x_{j}^{\left(  r\right)  }\right)  _{j=1}^{\infty}\in
\gamma_{s_{r}}\left(  E_{r}\right)  $ and $\left\Vert \left(  x_{j}^{\left(
r\right)  }\right)  _{j=1}^{\infty}\right\Vert _{\gamma_{s_{r}}\left(
E_{r}\right)  }\leq1,$ $r=1,...,m.$
\end{lemma}

As in the case of the Proposition \ref{esse} we have a characterization for
the operators in $\Pi_{\gamma_{\left(  s;s_{1},...,s_{m}\right)  }}%
^{ev}\left(  E_{1},...,E_{m};F\right)  .$ The argument used in the proof is an
adaptation of an original argument due to M.C.$\ $Matos (see \cite{matos}):

\begin{theorem}
\label{4.1}For $T\in\mathcal{L}\left(  E_{1},...,E_{m};F\right)  $, the
following statements are equivalent$:$

$\left(  i\right)  $ $T\in\Pi_{\gamma_{\left(  s;s_{1},...,s_{m}\right)  }%
}^{ev}\left(  E_{1},...,E_{m};F\right)  ;$

$\left(  ii\right)  ~$There is a constant $C>0$ such that
\begin{align*}
&  \left\Vert \left(  T\left(  b_{1}+x_{j}^{\left(  1\right)  },...,b_{m}%
+x_{j}^{\left(  m\right)  }\right)  -T\left(  b_{1},...,b_{m}\right)  \right)
_{j=1}^{n}\right\Vert _{\gamma_{s}\left(  F\right)  }\\
&  \leq C\left(  \left\Vert b_{1}\right\Vert +\left\Vert \left(
x_{j}^{\left(  1\right)  }\right)  _{j=1}^{n}\right\Vert _{\gamma_{s_{1}%
}\left(  E_{1}\right)  }\right)  \cdots\left(  \left\Vert b_{m}\right\Vert
+\left\Vert \left(  x_{j}^{\left(  m\right)  }\right)  _{j=1}^{n}\right\Vert
_{\gamma_{s_{m}}\left(  E_{m}\right)  }\right)  ,
\end{align*}
for all positive integer $n.$

$\left(  iii\right)  $ There is a constant $C>0$ such that
\begin{align}
&  \left\Vert \left(  T\left(  b_{1}+x_{j}^{\left(  1\right)  },...,b_{m}%
+x_{j}^{\left(  m\right)  }\right)  -T\left(  b_{1},...,b_{m}\right)  \right)
_{j=1}^{\infty}\right\Vert _{\gamma_{s}\left(  F\right)  }\label{d}\\
&  \leq C\left(  \left\Vert b_{1}\right\Vert +\left\Vert \left(
x_{j}^{\left(  1\right)  }\right)  _{j=1}^{\infty}\right\Vert _{\gamma_{s_{1}%
}\left(  E_{1}\right)  }\right)  \cdots\left(  \left\Vert b_{m}\right\Vert
+\left\Vert \left(  x_{j}^{\left(  m\right)  }\right)  _{j=1}^{\infty
}\right\Vert _{\gamma_{s_{m}}\left(  E_{m}\right)  }\right)  ,\nonumber
\end{align}
for all $\left(  b_{1},...,b_{m}\right)  \in E_{1}\times\cdots\times E_{m}$
and $\left(  x_{j}^{\left(  r\right)  }\right)  _{j=1}^{\infty}\in
\gamma_{s_{r}}\left(  E_{r}\right)  ,r=1,...,m.$ In addition, the smallest of
the constants $C$ satisfying (\ref{d}), denoted by $\pi^{ev}\left(
\cdot\right)  $, defines a norm on $\Pi_{\gamma_{\left(  s;s_{1}%
,...,s_{m}\right)  }}^{ev}\left(  E_{1},...,E_{m};F\right)  $.
\end{theorem}

\begin{proof}
$\left(  iii\right)  \Rightarrow\left(  ii\right)  $ and $\left(  iii\right)
\Rightarrow\left(  i\right)  $ are immediate.

$\left(  ii\right)  \Rightarrow\left(  iii\right)  $ is a straightforward
consequence of \textit{(P1).}

$\left(  i\right)  \Rightarrow\left(  ii\right)  .$ Define $G_{r}=E_{r}%
\times\gamma_{s_{r}}\left(  E_{r}\right)  $, $r=1,...,m$, with the $\ell_{1}$norm.

For the sake of simplicity, we write, for all $k=1,...,m,$
\[
\left(  b_{k},\left(  x_{j}^{\left(  k\right)  }\right)  _{j=1}^{n}\right)
:=\left(  b_{k},x_{n}^{\left(  k\right)  }\right)
\]
and consider the continuous $m$-linear operator%
\[
\Phi\left(  T\right)  :G_{1}\times\cdots\times G_{m}\longrightarrow\gamma
_{s}\left(  F\right)
\]
given by%
\[
\left(  \left(  b_{1},x_{\infty}^{\left(  1\right)  }\right)  ,...,\left(
b_{m},x_{\infty}^{\left(  m\right)  }\right)  \right)  \longrightarrow\left(
T\left(  b_{1}+x_{j}^{\left(  1\right)  },...,b_{m}+x_{j}^{\left(  m\right)
}\right)  -T\left(  b_{1},...,b_{m}\right)  \right)  _{j=1}^{\infty}.
\]

Let%
\begin{align*}
&  F_{k,\left(  x_{j}^{\left(  1\right)  }\right)  _{j=1}^{\infty},...,\left(
x_{j}^{\left(  m\right)  }\right)  _{j=1}^{\infty}}\\
&  :=\left\{  \left(  b_{1},...,b_{m}\right)  \in E_{1}\times\cdots\times
E_{m};\left\Vert \Phi\left(  T\right)  \left(  \left(  b_{1},x_{\infty
}^{\left(  1\right)  }\right)  ,...,\left(  b_{m},x_{\infty}^{\left(
m\right)  }\right)  \right)  \right\Vert _{\gamma_{s}\left(  F\right)  }\leq
k\right\}
\end{align*}
with $\left(  x_{j}^{\left(  r\right)  }\right)  _{j=1}^{\infty}\in
B_{\gamma_{s_{r}}\left(  E_{r}\right)  },r=1,...,m.$ For all positive integers
$n$, let%
\begin{align*}
&  F_{k,\left(  x_{j}^{\left(  1\right)  }\right)  _{j=1}^{n},...,\left(
x_{j}^{\left(  m\right)  }\right)  _{j=1}^{n}}\\
&  =\left\{  \left(  b_{1},...,b_{m}\right)  \in E_{1}\times\cdots\times
E_{m};\left\Vert \Phi\left(  T\right)  \left(  \left(  b_{1},x_{n}^{\left(
1\right)  }\right)  ,...,\left(  b_{m},x_{n}^{\left(  m\right)  }\right)
\right)  \right\Vert _{\gamma_{s}\left(  F\right)  }\leq k\right\}  .
\end{align*}
Using \textit{(P1) }we have%
\[
F_{k,\left(  x_{j}^{\left(  1\right)  }\right)  _{j=1}^{\infty},...,\left(
x_{j}^{\left(  m\right)  }\right)  _{j=1}^{\infty}}=\bigcap_{n\in\mathbb{N}%
}F_{k,\left(  x_{j}^{\left(  1\right)  }\right)  _{j=1}^{n},...,\left(
x_{j}^{\left(  m\right)  }\right)  _{j=1}^{n}}.
\]
For all $\left(  x_{j}^{\left(  r\right)  }\right)  _{j=1}^{\infty}\in
B_{\gamma_{s_{r}}\left(  E_{r}\right)  }$, $r=1,...,m$ and fixed $k$, consider%
\[
D_{k}:E_{1}\times\cdots\times E_{m}\rightarrow\lbrack0,\infty),
\]
given by%
\[
D_{k}\left(  b_{1},...,b_{m}\right)  =\left\Vert \left(  T\left(  b_{1}%
+x_{j}^{\left(  1\right)  },...,b_{m}+x_{j}^{\left(  m\right)  }\right)
-T\left(  b_{1},...,b_{m}\right)  \right)  _{j=1}^{k}\right\Vert _{\gamma
_{s}\left(  F\right)  }.
\]

Let us see that $D_{k}$ is continuous for all $k\in\mathbb{N}.$ Let $\left(
b_{n}\right)  _{n=1}^{\infty}=\left(  \left(  b_{n}^{1},...,b_{n}^{m}\right)
\right)  _{n=1}^{\infty}$ be a sequence in $E_{1}\times\cdots\times E_{m}$
converging to $b=\left(  b_{1},...,b_{m}\right)  $. Note that%
\begin{align}
&  \lim_{n\rightarrow\infty}D_{k}\left(  b_{n}^{1},...,b_{n}^{m}\right)
\label{mmm2}\\
&  =\left\Vert \lim_{n\rightarrow\infty}\left(  T\left(  b_{n}^{1}%
+x_{j}^{\left(  1\right)  },...,b_{n}^{m}+x_{j}^{\left(  m\right)  }\right)
-T\left(  b_{n}^{1},...,b_{n}^{m}\right)  \right)  _{j=1}^{k}\right\Vert
_{\gamma_{s}\left(  F\right)  }.\nonumber
\end{align}
Let us write
\[
h_{n}^{j}:=\left(  T\left(  b_{n}^{1}+x_{j}^{\left(  1\right)  },...,b_{n}%
^{m}+x_{j}^{\left(  m\right)  }\right)  -T\left(  b_{n}^{1},...,b_{n}%
^{m}\right)  \right)
\]
and%
\[
h_{j}:=\left(  T\left(  b_{1}+x_{j}^{\left(  1\right)  },...,b_{m}%
+x_{j}^{\left(  m\right)  }\right)  -T\left(  b_{1},...,b_{m}\right)  \right)
.
\]
Since $T$ is bounded, then for each $j\in\{1,...,k\},$ we have $h_{n}%
^{j}\underset{n\rightarrow\infty}{{\LARGE \rightarrow}}h_{j}$, i.e.,
\[
\left\Vert h_{n}^{j}-h_{j}\right\Vert _{F}<\frac{\epsilon}{k}\text{, for all
}n\geq n_{0}^{j}%
\]
for all $j\in\{1,...,k\}.$ Now, since%
\begin{align*}
&  \left\Vert \left(  h_{n}^{1}-h_{1},...,h_{n}^{k}-h_{k},0,...\right)
\right\Vert _{\gamma_{s}\left(  F\right)  }\\
&  \leq\left\Vert \left(  h_{n}^{1}-h_{1},0,...\right)  \right\Vert
_{\gamma_{s}\left(  F\right)  }+...+\left\Vert \left(  0,....,0,h_{n}%
^{k}-h_{k},0,...\right)  \right\Vert _{\gamma_{s}\left(  F\right)  }\\
&  \overset{(P2)}{=}\left\Vert h_{n}^{1}-h_{1}\right\Vert _{F}+...+\left\Vert
h_{n}^{k}-h_{k}\right\Vert _{F}\\
&  \leq\epsilon,
\end{align*}

for all $n\geq\tilde{n}_{0}$, with $\tilde{n}_{0}=\max\{n_{0}^{j}\},$ then we
have
\[
\left\Vert \left(  h_{n}^{1},...,h_{n}^{k}\right)  -\left(  h_{1}%
,...,h_{k}\right)  \right\Vert _{\gamma_{s}\left(  F\right)  }\leq
\epsilon\text{ }%
\]
for all $n\geq\tilde{n}_{0}.$ So, from (\ref{mmm2}) we obtain
\[
\lim_{n\rightarrow\infty}D_{k}\left(  b_{n}^{1},...,b_{n}^{m}\right)
=D_{k}\left(  b_{1},...,b_{m}\right)  .
\]
Now we note that $F_{k,\left(  x_{j}^{\left(  1\right)  }\right)  _{j=1}%
^{n},...,\left(  x_{j}^{\left(  m\right)  }\right)  _{j=1}^{n}}$ is closed
since%
\[
F_{k,\left(  x_{j}^{\left(  1\right)  }\right)  _{j=1}^{n},...,\left(
x_{j}^{\left(  m\right)  }\right)  _{j=1}^{n}}=D_{n}^{-1}\left(  \left[
0,k\right]  \right)  ,
\]
and thus $F_{k,\left(  x_{j}^{\left(  1\right)  }\right)  _{j=1}^{\infty
},...,\left(  x_{j}^{\left(  m\right)  }\right)  _{j=1}^{\infty}}$ is closed.

Let
\[
F_{k}:=%
%TCIMACRO{\dbigcap }%
%BeginExpansion
{\displaystyle\bigcap}
%EndExpansion
F_{k,\left(  x_{j}^{\left(  1\right)  }\right)  _{j=1}^{\infty},...,\left(
x_{j}^{\left(  m\right)  }\right)  _{j=1}^{\infty}},
\]
where the intersection is taken over all $\left(  x_{j}^{\left(  r\right)
}\right)  _{j=1}^{\infty}\in B_{\gamma_{s_{r}}\left(  E_{r}\right)
},~r=1,...,m$. From Lemma \ref{9.2}, we have%
\[
E_{1}\times\cdots\times E_{m}=\bigcup\limits_{k\in\mathbb{N}}F_{k},
\]

and from Baire Category Theorem there is an interior point $\left(
b_{1},...,b_{m}\right)  $ of some $F_{k_{0}}.$ Hence, there is a
$0<\varepsilon<1$ such that%
\begin{equation}
\left\Vert \Phi\left(  T\right)  \left(  \left(  c_{1},\left(  x_{j}^{\left(
1\right)  }\right)  _{j=1}^{\infty}\right)  ,...,\left(  c_{m},\left(
x_{j}^{\left(  m\right)  }\right)  _{j=1}^{\infty}\right)  \right)
\right\Vert _{\gamma_{s}\left(  F\right)  }\leq k_{0} \label{2fevz}%
\end{equation}
whenever $\left\Vert c_{r}-b_{r}\right\Vert <\varepsilon$ and $\left(
x_{j}^{\left(  r\right)  }\right)  _{j=1}^{\infty}\in B_{\gamma_{s_{r}}\left(
E_{r}\right)  },~r=1,...,m$. If
\[
\left\Vert \left(  v_{r},\left(  x_{j}^{(r)}\right)  _{j=1}^{\infty}\right)
\right\Vert <\varepsilon
\]
for all $r=1,...,m$, we have%
\[
\left\Vert v_{r}\right\Vert <\varepsilon\text{ and }\left\Vert \left(
x_{j}^{(r)}\right)  _{j=1}^{\infty}\right\Vert _{\gamma_{s_{r}}\left(
E_{r}\right)  }<\varepsilon<1.
\]
So, using (\ref{2fevz}), it follows that
\[
\left\Vert \Phi\left(  T\right)  \left[  \left(  b_{1}+v_{1},\left(
x_{j}^{(1)}\right)  _{j=1}^{\infty}\right)  ,...,\left(  b_{m}+v_{m},\left(
x_{j}^{\left(  m\right)  }\right)  _{j=1}^{\infty}\right)  \right]
\right\Vert _{\gamma_{s}\left(  F\right)  }\leq k_{0}.
\]
We conclude that $\Phi\left(  T\right)  $ is bounded in the open ball of
radius $\varepsilon$ centered at
\[
\left(  \left(  b_{1},\left(  0\right)  _{j=1}^{\infty}\right)  ,...,\left(
b_{m},\left(  0\right)  _{j=1}^{\infty}\right)  \right)  \in G_{1}\times
\cdots\times G_{m},
\]
and we conclude that $\Phi\left(  T\right)  $ is continuous.

Note that from the definition of $\Phi(T)$ we have
\[
\pi^{ev}\left(  T\right)  =\left\Vert \Phi\left(  T\right)  \right\Vert
\]
and the infimum in (\ref{d}) is attained.
\end{proof}

Using standard arguments, now adapted to our abstract setting, one can easily
conclude that $\left(  \Pi_{\gamma_{\left(  s;s_{1},...,s_{m}\right)  }}%
^{ev}\left(  E_{1},...,E_{m};F\right)  ,\pi^{ev}\right)  $ is a Banach space.
Moreover, considering the notion of ideals of multilinear mappings in the
sense of \cite{Flo}, we have

\begin{theorem}
\label{mp}$(\Pi_{\gamma_{\left(  s;s_{1},...,s_{m}\right)  }}^{ev},\pi^{ev})$
is a Banach ideal of $m$-linear mappings.
\end{theorem}

\section{Applications\label{appl}}

In this section we show that several results related to everywhere summability
found on the literature can be regarded as particular cases of our abstract approach:

\begin{itemize}
\item \textbf{Absolutely summing multilinear operators}
\end{itemize}

The notion of everywhere absolutely summing operators was introduced by M.C.
Matos \cite{matos} and explored in \cite{pellscan}.

We just need to consider
\[
\gamma_{s_{k}}\left(  E_{k}\right)  =\ell_{q_{k}}^{u}\left(  E_{k}\right)
\text{, for all }k=1,...,m
\]
and%
\[
\gamma_{s}\left(  F\right)  =\ell_{p}\left(  F\right)  ,
\]
with $\frac{1}{p}\leq\frac{1}{q_{1}}+...+\frac{1}{q_{m}}$, to recover the
definitions from \cite[Pag. 221]{pellscan}. Since the above sequence spaces
satisfy \textit{(P1)} and\textit{ (P2)}, our approach recovers the results
from \cite[\textit{Theorem }4.1]{pellscan} and \cite[\textit{Lemma
}9.2\textit{\ }and\textit{ Proposition }9.4.]{studia}\textit{.}

\begin{itemize}
\item \textbf{Almost summing multilinear operators}
\end{itemize}

The notion of everywhere almost summing multilinear operators was introduced
in \cite{Darchiv} and later explored in \cite{pelrib}. In our abstract setting
it suffices to consider
\[
\gamma_{s_{k}}\left(  E_{k}\right)  =\ell_{p_{k}}^{u}\left(  E_{k}\right)
,~\text{for all }k=1,...,m,
\]
and%
\[
\gamma_{s}\left(  F\right)  =Rad\left(  F\right)
\]

to recover the definitions from \cite[Section 3]{pelrib}. We note that
$Rad\left(  F\right)  $ satisfies \textit{(P1)} and \textit{(P2). }In
fact,\textit{ (P2) }is easily proved, and to obtain
\[
\left\Vert \left(  x_{k}\right)  _{k=1}^{\infty}\right\Vert _{Rad(F)}=\sup
_{n}\left\Vert \left(  x_{k}\right)  _{k=1}^{n}\right\Vert _{Rad(F)}%
\]
we use the Contraction Principle and Hahn-Banach Theorem.\textit{ }So, we
recover \cite[\textit{Theorem }3.7\textit{. Theorem }4.4. ]{pelrib}.

\bigskip

\begin{itemize}
\item \textbf{Cohen strongly summing multilinear operators}
\end{itemize}

The notion of everywhere Cohen strongly $p$-summing multilinear operators,
denoted by $\mathcal{L}_{Coh,p}\left(  E_{1},...,E_{m};F\right)  $, was
introduced in \cite{jamtese, campos}. We consider%
\[
\gamma_{s_{k}}\left(  E_{k}\right)  =\ell_{p}\left(  E_{k}\right)  \text{, for
all }k=1,...,m
\]
and%
\[
\gamma_{s}\left(  F\right)  =\ell_{p}\left\langle F\right\rangle \text{,}%
\]
to recover the definitions from \cite[Definition 6.1.1]{jamtese}. Since
$\ell_{p}\left\langle F\right\rangle $ satisfies \textit{(P1)} and\textit{
(P2),} we recover \cite[\textit{Proposition} 6.1.10, \textit{Theorem} 6.1.12,
and \textit{Proposition} 6.2]{jamtese}. We stress that our approach also shows
that the ideal of everywhere Cohen strongly $p$-summing multilinear operators
is complete; for technical reasons this result was not proved in
\cite{jamtese}.

\section{Dvoretzky-Rogers type theorems}

Until 1950 it was not known if every infinite-dimensional Banach space had an
unconditionally convergent series which fails to be absolutely convergent;
this question is due to Banach \cite[p. 40]{Banach32}. In 1950, A. Dvoretzky
and C.A. Rogers \cite{DR} gave a positive answer to this question:

\bigskip

\textbf{Theorem} (Dvoretzky-Rogers, 1950). The unconditionally convergent
series and absolutely summing convergent series coincide in a Banach space $E
$ if and only if $\dim E=\infty.$

\bigskip

In this Section we remark that if there is a Dvoretzky-Rogers type theorem for
absolutely $\gamma_{\left(  s;s_{1}\right)  }$-summing linear operators, then
this result is inherited by the multilinear maps.

So, we will suppose that the following holds true:%
\[
E\text{ is finite dimensional}\Longleftrightarrow\Pi_{\gamma_{\left(
s,s_{1}\right)  }}\left(  E;E\right)  =\mathcal{L}\left(  E;E\right)
\]

From now on we will tacitly suppose that $\Pi_{\gamma_{\left(  s;s_{1}%
,...,s_{m}\right)  }}\left(  E_{1},...,E_{m};F\right)  $ is nontrivial, i.e.,
contains the $m$-linear finite type operators.

With some effort one can obtain the following Dvoretzky-Rogers type theorem in
our abstract framework which, as we will mention later, recovers various
particular results:

\begin{theorem}
[Dvoretzky-Rogers Theorem for $\gamma$-summing operators]\label{tal}Let $E$ a
Banach space and $m\geq2.$ The following statements are equivalent:

$i)$ $E$ is infinite-dimensional

$ii)$ $\Pi_{\gamma_{\left(  s,s_{1}\right)  }}^{a}\left(  ^{m}E;E\right)
\not =\mathcal{L}\left(  ^{m}E,E\right)  $ for every $\left(  a_{1}%
,...,a_{m}\right)  \in E^{m}$ with either $a_{i}\not =0$ for every $i$ or
$a_{i}=0$ for only one $i.$

$iii)$ $\Pi_{\gamma_{\left(  s,s_{1}\right)  }}^{a}\left(  ^{m}E;E\right)
\not =\mathcal{L}\left(  ^{n}E,E\right)  $ for some $\left(  a_{1}%
,...,a_{m}\right)  \in E^{m}$ with either $a_{i}\not =0$ for every $i$ or
$a_{i}=0$ for only one $i.$
\end{theorem}

\begin{corollary}
Let $E$ be an infinite-dimensional Banach space, $a=\left(  a_{1}%
,...,a_{m}\right)  \in E^{m}$, $m\geq2.$ If $\ \Pi_{\gamma_{\left(
s,s_{1}\right)  }}^{a}\left(  ^{m}E;E\right)  =\mathcal{L}\left(
^{m}E;E\right)  $, then card$~\left\{  i:a_{i}=0\right\}  \geq2$. In
particular, if $\ \Pi_{\gamma_{\left(  s,s_{1}\right)  }}^{a}\left(
^{2}E;E\right)  =\mathcal{L}\left(  ^{2}E;E\right)  $ then $a=0$.
\end{corollary}

The above results recover the following particular cases:

\begin{itemize}
\item \textbf{Absolutely summing multilinear operators.} See
\cite[\textit{Theorem }3.2]{pellscan}.

\item \textbf{Almost summing multilinear operators}. See
\cite[\textit{Theorem} 3.2]{pelrib}.

\item \textbf{Cohen strongly summing multilinear operators.} See
\cite[\textit{Theorem} 6.1.9]{jamtese}.
\end{itemize}

\end{document}